\documentclass[letterpaper, 10 pt, conference]{ieeeconf}

\IEEEoverridecommandlockouts                              

\usepackage{cite}
\usepackage{textcomp}
\usepackage{amsfonts}
\usepackage{amsmath}
\usepackage{graphicx}
\usepackage{subcaption}
\usepackage{mdframed}
\usepackage{amssymb}
\usepackage[mathcal]{euscript}

\newcommand{\blue}[1]{#1}

\newcommand{\ip}[2]{\langle #1, #2 \rangle}
\newcommand{\mcl}[1]{\mathcal{ #1}}
\newcommand{\mbf}[1]{\mathbf{ #1}}

\let\bbl\Bigl

\let\bbr\Bigr

\newlength{\myitemsep} 
\newlength{\mylistsep} 
\newlength{\mymathsep} 
\newlength{\myparsep} 
\newlength{\mysubsecsep} 
\newlength{\mysubsecpresep} 
\newlength{\mysecsep} 
\newlength{\mysecpresep} 

\newcommand{\R}{\mathbb{R}}

\newcommand{\bmat}[1]{\begin{bmatrix}#1\end{bmatrix}}
\newcommand{\norm}[1]{\lVert{#1}\rVert}

\newcommand{\half}{\frac{1}{2}}

\newtheorem{thm}{Theorem}
\newtheorem{example}{Example}%
\newtheorem{note}{Note}%

\newtheorem{defn}{Definition}

\newtheorem{lem}[thm]{Lemma}

\begin{document}

\title{\LARGE \bf
Hierarchical Stability and Lyapunov Conditions for Linear PDEs
}

\author{Matthew~M.~Peet
\thanks{M. Peet is with the School for the Engineering of Matter, Transport and Energy, Arizona State University. e-mail: {\tt \small mpeet@asu.edu }} \thanks{This work was supported by the National Science Foundation under grants No. 2429973 and 2337751.}
}

\maketitle
\thispagestyle{empty}

\begin{abstract}
Unlike ordinary differential equations (ODEs), linear partial differential equations (PDEs) admit multiple non-equivalent notions of stability. This variety makes interpretation of Lyapunov stability results challenging. \blue{To simplify this interpretation, we propose a unified framework for hierarchical classification of notions of stability and Lyapunov conditions. To do this, for classically well-posed PDEs with admissible boundary conditions}, we define a fundamental state on $L_2$ corresponding to the minimal information needed to uniquely forward propagate the solution. Stability notions and Lyapunov functions are then defined in terms of this fundamental state. This gives rise to a hierarchy of stability notions, the weakest being fundamental state to PDE state stability. Other stability notions and Lyapunov conditions may then be interpreted relative to this weakest notion. Hierarchies are established for: Lyapunov, exponential and finite-energy stability. Sufficient Lyapunov conditions are defined in terms of operator inequalities. Illustrative examples and computational tools are provided.\vspace{-3mm}
\end{abstract}
\vspace{-2mm}

\section{Introduction}\vspace{-2mm}
Recent years have seen an increase in methods for stability analysis and control of partial differential equations (PDEs) without discretization. Often, these methods rely on Lyapunov functions to establish convergence of the system state, bound $L_2$ gain, or design observers and controllers. Unlike ODEs, for which we have a well-established universal state-space representation, PDEs take a variety of forms -- with no clear and consistent notion of state space. As a result, basic definitions of stability vary from PDE to PDE and, even for a given PDE, will vary with the set of imposed boundary conditions. Semigroup theory~\cite{curtain_book} provides definitions of state which apply to a large class of PDEs. However this approach simply shifts the problem into various definitions of inner product, domain, generator, et c.. 
The use of Lyapunov functions further complicates the problem. While it is accepted that a Lyapunov function should be positive, bounded, and non-increasing, the form taken by these upper and lower bounds is not consistent. As a result, even for a given PDE with a given set of boundary conditions, we may obtain a variety of non-equivalent stability results depending on how the Lyapunov function is defined and bounded.

\blue{To illustrate, consider a well-studied PDE: the wave equation $\ddot{\mbf u}=\partial_s^2\mbf u$ with pinned boundary conditions $ \mbf u(t,0)= \mbf u(t,1)=0$. This PDE is not stable in the $L_2\times L_2$ sense in that there exists no $C\ge 0$ such that $\norm{\mbf u(t)}_{L_2}+\norm{\dot{\mbf u}(t)}_{L_2}\le C (\norm{\mbf u(0)}_{L_2}+\norm{\dot{\mbf u}(0)}_{L_2})$ for all $t \ge 0$. However, this PDE is stable in the $ H^1 \times L_2$ sense, in that for $V(u)=\norm{\dot u}_{L_2}+\norm{\partial_s u}_{L_2}$, $\dot V=0$ and hence $\norm{\partial_s\mbf u(t)}_{L_2}+\norm{\dot{\mbf u}(t)}_{L_2} \le \norm{\partial_s\mbf u(0)}_{L_2}+\norm{\dot{\mbf u}(0)}_{L_2}$} 
\footnote{However, this tells us nothing about the actual PDE state, $\mbf u(t)$. \blue{Fortunately, we know that $\norm{\mbf u}\le C\norm{\partial_s\mbf u}$ where $C$ is the Poincar\'e constant and hence $\norm{\mbf u(t)}_{L_2}+\norm{\dot{\mbf u}(t)}_{L_2} \le C\norm{\partial_s\mbf u(0)}_{L_2}+\norm{\dot{\mbf u}(0)}_{L_2}$.} However, as indicated, this bound cannot be strengthened to the first, more tradition notion of stability or well-posedness in $\mbf u$. }.
This illustration shows that obvious notions of stability often do not hold for PDEs (beam equations have the same problem). Nonetheless, it seems reasonable that we would want to be able to include wave equations in any Lyapunov framework for analysis and control of PDEs. Clearly, then, we must allow for some weaker notions of stability in the analysis and control of PDEs. 

The need for a weaker notion of stability (and well-posedness) is well-known. Such notions for the wave equation and can be found in~\cite{ha1994stability}. Other examples where stability depends on the choice of norm can be found in~\cite{beck2006_thesis}. For a broader class of PDE, the framework developed in~\cite{mironchenko2019non,jacob_2018,jacob2020noncoercive} for non-coercive Lyapunov functions weakens the  positivity condition $V(\mbf u)\ge \epsilon \norm{\mbf u}^2$ to $V(\mbf u)>0$ for all $\mbf u \neq 0$. A strict negativity condition on the derivative of the Lyapunov function is then used to upper bound the solutions. However, the difficulty in proving stability of PDEs such as the wave equation lies not in positivity of the Lyapunov function, but in the upper bound. For example the Lyapunov function used for the  wave equation discussed above is coercive in $\norm{\mbf u}$. However, it is not upper bounded by $C\norm{\mbf u}$ for any $C>0$. 

In this context, \blue{the contribution of this paper is primarily methodical organization and unification. We classify several properties of Lyapunov functions typically used in analysis of PDEs. These properties are divided into 3 categories: positivity (lower bounded), upper bounded, and negativity of the derivative. We then interpret those properties systematically using a hierarchy of notions of stability. The difficulty, as suggested above, is that PDEs, boundary conditions, and Lyapunov functions vary substantially from case to case. To address this problem, this paper is structured around an operator or integral formulation of the PDE referred to as a Partial Integral Equation (PIE) or fundamental state representation~\cite{shivakumar_TAC2024}. Specifically, if the boundary conditions of a PDE are suitably well-defined, there exists a bijection (Green's function) from $L_2$ to the set of sufficiently regular functions which satisfy the associated set of boundary conditions. }

For example, consider the heat equation $\partial_t\mbf u=\partial_s^2\mbf u$ where solutions are constrained to lie in some domain \blue{(similar to $D(A)$ for semigroups)}, e.g. $X:=\{\mbf u \in H^2\;:\; \mbf u(0)=\mbf u(1)=0\}$. Then if we define  $\mcl T$ (Green's function) as \vspace{-2mm}
\[
(\mcl T \mbf x):=\int_0^1G(s,\theta)\mbf x(\theta)d\theta\quad G(s,\theta)=\begin{cases}\theta (s-1)&\theta\le s\\s (\theta-1) & s \le \theta,
\end{cases}
\vspace{-2mm}
\]

\noindent we have that $\mcl T : L_2\rightarrow X$, $\mcl T\partial_s^2 \mbf u=\mbf u$ and $\partial_s^2 \mcl T \mbf x=\mbf x$ for any $\mbf u \in X$ and $\mbf x \in L_2$. We refer to $\mbf x:=\partial_s^2\mbf u$ as the fundamental state. Using this bijection, the dynamics of a \blue{classically well-posed PDE} may be equivalently
\footnote{For notational simplicity, we assume sufficient regularity for commutation of spatial and temporal derivatives. Otherwise we use $\partial_t (\mcl T\mbf x)=\mcl A \mbf x$.}
 written as $\mcl T\dot{\mbf x}=\mcl A \mbf x$. Of course, for the heat equation, this is simply\footnote{More generally, we may always write the dynamics of a PDE $\dot u=Au$ as $\mcl T \dot{\mbf x}=\mbf x$ where $\mcl T=A^{-1}$ on $X$. However, in this case, the kernels of operator $\mcl T$ may not be polynomial and require analytic solution of certain differential equations. If such an operator \textit{is} known, however, the results of this paper can be applied using this $\mcl T$ and $\mcl A=I$.} $\mcl T\dot{\mbf x}=\mbf x$ and $\mcl T=(\partial_s^2)^{-1}$ on the domain $X$. An equation of the form $\mcl T\dot{\mbf x}=\mcl A\mbf x$ where $\mcl T$ and $\mcl A$ have the structure given above is a Partial Integral Equations (PIEs). An operator $\mcl T$ of this type is a Partial Integral (PI) operator and the set of PI operators define a linear algebra, with the set of PI operators with polynomial kernels being a sub-algebra. It has been shown that any PDE with suitably well-defined domain, $X$, admits such a representation with polynomial kernels~\cite{peet_2021AUTb,jagt_2025AUT,shivakumar_TAC2024}. 

Our approach, then, is to allow for classification of bounds on a Lyapunov function in terms of both the PDE state and the fundamental (PIE) state. Next, we classify notions of stability in terms of these bounds -- focusing on two cases of Lyapunov stability: PDE stability ($\norm{\mbf u(t)}\le C \norm{\mbf u(0)}$) and the weaker notion of PIE to PDE stability ($\norm{\mbf u(t)}\le C \norm{\mbf x(0)}$). These notions are extended to exponential and finite-energy stability. This classification system then allows for interpretation of almost any PDE stability result which is based on the use of Lyapunov functions
.
Finally, to illustrate the application of these results, we express sufficient conditions for these various stability notions in terms of operator inequalities. These operator inequalities are applied to representative PDEs.

\vspace{-1mm}\vspace{-1mm}

\section{Notation}\label{sec:notation}\vspace{-1mm}
$L_2[\Omega]$ denotes the space of Lebesgue square integrable functions on $\Omega$. Unless otherwise stated, $\Omega=[0,1]$. We use the Sobolev space $H^n:=\{\mbf u\;:\;\partial_s^i \mbf u\in L_2,\; i \le n\}$. For $\mbf u \in H^n$, we use the shorthand $\mbf u_{s}:=\partial_s \mbf u$, $\mbf u_{ss}:=\partial_s^2 \mbf u$, et c.  $\mcl L(X)$ denotes the Banach space of bounded linear operators on Banach space $X$ with operator norm. Bolded lowercase (e.g. $\mbf  u$) indicates a function of a spatial variable. Calligraphic uppercase (e.g. $\mcl A$) indicates a bounded linear operator. For matrix $N$, the $i,j$th element is denoted $[N]_{i,j}:=N_{i,j}$.
\vspace{-1mm}\vspace{-1mm}

\section{Fundamental (PIE) State and Bijection $\mcl T$}\label{sec:fundamental_state}\vspace{-1mm}
\begin{defn}
\blue{$\mcl P$ is a Partial Integral (PI) operator if \vspace{-3mm}
\[
(\mcl P \mbf x)\blue{(s)}:=R_0(s)\mbf x(s)+\int_0^s\hspace{-2mm}R_1(s,\theta)\mbf x(\theta)d\theta+\int_s^1 \hspace{-2mm}R_2(s,\theta)\mbf x(\theta)d\theta\vspace{-1mm}
\]
with $R_0 \in L_\infty$ and $R_1,R_2\in L_2$. The vector space of all PI operators is denoted $\Pi_3$, which is a subspace of $\mcl L(L_2)$. We denote the subspace $\Pi_2:=\{\mcl P \in \Pi_3\;:\; R_0=0\} \subset \Pi_3$.}
\end{defn}
\blue{$\Pi_3$ and $\Pi_2$ are vector spaces and, when the $R_i$ are square, are linear composition *-subalgebras of $\mcl L(L_2)$.} The subspaces of $\Pi_3$ and $\Pi_2$ with $R_i$ polynomial are likewise *-subalgebras. Analytic formulae for composition and adjoint of $\mcl P$ in terms of the parameters $R_i$ can be found in~\cite{peet_2021AUTb}, Lemmas 4 and 6. Furthermore, for $\mcl P \in \Pi_3$ and $\mcl Q \in \Pi_2$, we have $\mcl{PQ},\mcl{QP}\in \Pi_2$. Extension to multivariate domains is the sum and composition of $\Pi_2,\Pi_3$ operators on each spatial dimension. Extension to mixed spatial domains such as $\mcl L(\R^n \times L_2[0,1] \times L_2[[0,1]^2])$ can be found in~\cite{jagt_2025AUT,shivakumar_TAC2024}. 

PDEs are typically defined in terms of $\dot{\mbf u}=A\mbf u$, where $A:X\rightarrow L_2$ is a differential operator and $\mbf u$ is constrained to lie in some subspace, $X$, of a Sobolev space. \blue{We presume the PDE is well-posed in the classical sense (See Definitions in~\cite{shivakumar_TAC2024}), so that for every $\mbf u_0\in X$, there exists $\mbf u(t)\in X$ for which $\dot{\mbf u}=A\mbf u$ and $\mbf u(0)=\mbf u_0$. If the PDE is well-posed and $X$ has the form in Eqn.~\eqref{eqn:gohberg_domain2}, then $A:X \rightarrow L_2$ is invertible on $X$ and $\mcl T:=A^{-1} \in \Pi_3$. This results in an equation of the form $\mcl T \dot{\mbf x}(t)=\mbf x(t)$ where we refer to $\mbf x(t) \in L_2$ as the fundamental state. However, solving for $A^{-1}$ explicitly requires solution of certain differential equations. In~\cite{peet_2021AUTb}, a framework was proposed whereby, roughly speaking, only the highest derivative, $D^\alpha$, in $A$ is inverted\footnote{If $A$ is not a differential operator, boundary conditions still require differentiation in order to obtain a bijection from $L_2$ to $X$.}. This results in a Partial Integral Equation (PIE) of the form $\mcl T \dot{\mbf x}(t)=\mcl A \mbf x(t)$ where $\mcl T=(D^{\alpha})^{-1}:L_2 \rightarrow X$ and $\mcl A=A\mcl T \in \Pi_3$. Operator $\mcl T$ then defines a one-to-one map between solutions of the PIE and PDE -- implying many problems of well-posedness~\cite{jagt_2026CDC} and stability~\cite{peet_2021AUTb} may be equivalently formulated in this framework. Significantly, analytic formulae can be obtained for the polynomial kernels, $R_1,R_2$ in $\mcl T$ -- implying that these operators can be readily manipulated and optimized using software such as PIETOOLS~\cite{PIETOOLS2025}. This construction is illustrated in the following lemma which concisely states the result in\cite{gohberg_book} and where we use $D^n =\partial_s^n$,}
\begin{align}
X=&\bbl\{g \in H^n[a,b] \;:\;  \qquad \text{for } i=1,\cdots,n\label{eqn:gohberg_domain2}\\[-2mm]
&\qquad \sum_{j=1}^n \alpha_{i,j}(D^{j-1}g)(a)+\sum_{j=1}^n \beta_{i,j}(D^{j-1}g)(b)=0\bbr\}\notag
\end{align}
and \vspace{-2mm}\vspace{-2mm}
{\small \begin{equation}
W_n(t)=\bmat{1 &t&\half t^2& \cdots &\frac{t^{n-1}}{(n-1)!}\\0&1&t&\cdots&\frac{t^{n-2}}{(n-2)!}\\\vdots &\vdots&\vdots&\vdots&\vdots\\0 &0&0&\cdots&t\\0 &0&0&\cdots&1}.
\end{equation}}
\begin{lem}\label{lem:gohberg_polynomial}
\blue{ Given $\alpha_{i,j},\beta_{i,j}\in \R$ and $X$ as defined in~\eqref{eqn:gohberg_domain2}, let $[N_a]_{i,j}=\alpha_{i,j}$, $[N_b]_{i,j}=\beta_{i,j}$. If $det(N_a+N_b W(b-a))\neq 0$, let}\vspace{-2mm}
 \begin{align}
(\mcl T\mbf x)(s)&=\int_a^b  G(s,\theta)\mbf x(\theta)  d\theta \\
G(s,\theta)&=
\begin{cases}
\mbf e_1(s-a)^T(I-P)\mbf e_n(a-\theta)&\theta < s\\
-\mbf e_1(s-a)^TP\mbf e_n(a-\theta)& s < \theta
\end{cases}
\end{align}
where $P=(N_a+N_b W(b-a))^{-1}N_bW(b-a)$ and \vspace{-2mm}{\small
\begin{equation*}
\mbf e_1(s)=\bmat{1 \hspace{-2mm}& s\hspace{-2mm}& \cdots \hspace{-2mm}& \frac{s^{n-1}}{(n-1)!}}^T, \qquad \mbf e_n(\theta)=\bmat{\frac{\theta^{n-1}}{(n-1)!}\hspace{-2mm}& \cdots\hspace{-2mm}& \theta \hspace{-2mm}&1}^T.
\end{equation*}}
Then for any $\mbf u \in X$ and $\mbf x \in L_2$, we have $D^n \mcl T \mbf x=\mbf x$, $\mcl T D^n \mbf u=\mbf u$ and $\mcl T \mbf x \in X$.
\end{lem}
We refer to $\mbf u$ as the PDE state and $\mbf x=D^\alpha \mbf u$ as the fundamental or PIE state. Stability notions and Lyapunov functions may then be defined in terms of the PIE state and, occasionally, interpreted in terms of the PDE state, $\mbf u=\mcl T\mbf x$. 

\section{Positivity, Boundedness, and Negativity}\label{sec:LF_bounds}
For each notion of stability which will be introduced in Sec.~\ref{sec:stability_notions}, we will associate a Lyapunov stability condition. Each of these Lyapunov conditions includes three parts: a lower bound on the function (positivity), an upper bound on the function, and a bound on the negativity of the derivative of the function. Before introducing notions of stability, we characterize variations on these 3 properties.
Specifically, we list possible properties of a Lyapunov function, $V$, in terms of PIE state, $\mbf x=D^{\alpha}\mbf u$ where the PDE state is obtained from the PIE state as $\mbf u=\mcl T\mbf x$ for some PI operator, $\mcl T$. This then implies that $\norm{\mbf u}_{L_2}\le \norm{\mcl T}\norm{\mbf x}_{L_2}$. We start with positivity properties, from weakest to strongest.
%


\begin{defn}[Lyapunov positivity]\label{def:Lyap_positive}
For a given operator $\mcl T:L_2 \rightarrow X$ and Lyapunov function, $V:L_2\rightarrow \R^+$ with $V(0)=0$, we say that $V$ is
\begin{enumerate}
\item \textbf{Positive semidefinite} if $V(\mbf x)\ge 0$ .
\item \textbf{PDE positive} if $V(\mbf x)\ge \epsilon \norm{\mcl T\mbf x}^2$ for some $\epsilon>0$.
\item \textbf{PIE positive} if  $V(\mbf x)\ge \epsilon \norm{\mbf x}^2$ for some $\epsilon>0$.
\end{enumerate}
These conditions are required to hold for all $\mbf x \in L_2$.
\end{defn}
The positivity types in Defn.~\ref{def:Lyap_positive} are hierarchical\footnote{For an explanation why ``Positive definite'' was not included in this list, see Note~\ref{note:non_coercive} in~\cite{peet_2026LCSS_arxiv}} -- PIE positive implies PDE positive implies positive definite implies positive semidefinite. The first holds because  $V(\mbf x)\ge \epsilon \norm{\mbf x}^2$ and $\norm{\mcl T \mbf x}\le \norm{\mcl T}\norm{\mbf x}$ imply 
$
V(\mbf x)\ge \epsilon \norm{\mbf x}^2\ge \frac{\epsilon}{\norm{\mcl T^2}} \norm{\mcl T\mbf x}^2.
$ 
However, PDE positivity does not imply PIE positivity, since there is no $C$ such that $\norm{\mbf x}\le C \norm{\mcl T \mbf x}$. 

For upper bounds, from weak to strong, we have

\begin{defn}[Lyapunov bounds]\label{def:Lyap_bounded}
Given a Lyapunov function, $V:L_2\rightarrow \R^+$ with $V(0)=0$, we say that $V$ is
\begin{enumerate}
\item \textbf{PIE bounded} if $V(\mbf x)\le C \norm{\mbf x}^2$ for some $C>0$.
\item \textbf{PDE bounded} if  $V(\mbf x)\le C \norm{\mcl T\mbf x}^2$ for some $C>0$.
\end{enumerate}
Here the inequalities must hold for and all $\mbf x \in L_2$.
\end{defn}

Any Lyapunov function of the form $V(\mbf x)=\ip{\mbf x}{\mcl P \mbf x}$ with PI operator $\mcl P\in \Pi_3$ is PIE bounded. Furthermore, PDE bounded implies PIE bounded since $\norm{\mcl T \mbf x} \le \norm{\mcl T}\norm{\mbf x}$.


For negativity, we presume $V(\mbf x)$ is sufficiently regular so that, for a given $\mcl T \dot{\mbf x}=\mcl A \mbf x$, we may define $\dot V(\mbf x)$. 

\begin{defn}[Lyapunov Derivative Conditions]\label{def:Lyap_negative}
Given $\mcl T ,\mcl A \in \Pi_3$ and $V:L_2\rightarrow \R^+$, suppose there exists $\dot V$ such that $\dot V(\mbf x)=\lim_{h\rightarrow 0^+}\frac{V(\hat{\mbf x}(t+h))-V(\mbf x)}{h}$ for any $\hat{\mbf x}(t)$ such that $\mcl T \dot{\hat{\mbf x}}(t)=\mcl A \hat{\mbf x}(t)$ with $\hat{\mbf x}(0)=\mbf x$. Then we say that $\dot V$ is
\begin{enumerate}
\item \textbf{Negative semidefinite} if $\dot V(\mbf x)\le 0$
\item \textbf{PDE negative} if $\dot V(\mbf x)\le -\alpha \norm{\mcl T\mbf x}^2$
\item \textbf{PIE negative} if  $\dot V(\mbf x)\le - \alpha \norm{\mbf x}^2$
\item \textbf{Lyapunov Negative} if $\dot V(\mbf x)\le - \alpha V(\mbf x)$
\end{enumerate}
for some $\alpha>0$. Inequalities must hold for all $\mbf x \in L_2$.
\end{defn}

Apart from ``Lyapunov negative'', negativity conditions are ordered from weakest to strongest -- i.e.  PIE negative implies PDE negative implies negative semidefinite. The first holds since $-\norm{\mbf x}^2\le -\frac{1}{\norm{\mcl T}^2}\norm{\mcl T\mbf x}^2$. The strength of ``Lyapunov negativity'' relative to PIE and PDE negativity depends on positivity and boundedness of $V$. Specifically, Lyapunov negativity implies PDE or PIE negativity if $V$ is PDE or PIE positive, respectively. Likewise, PDE or PIE negativity implies Lyapunov negativity if $V$ is PDE or PIE, respectively.\vspace{-1mm}

\section{Stability Notions and Lyapunov Functions}\label{sec:stability_notions}\vspace{-1mm}
\blue{In this section, we define 4 notions of stability and, for each notion of stability, associate a sufficient condition in terms of existence of a Lyapunov function which satisfies a set of bounds as defined in Sec.~\ref{sec:LF_bounds}. For consistency, these notions of stability and Lyapunov conditions are all expressed in terms of a PIE $\mcl T \dot{\mbf x}=\mcl A \mbf x$ with fundamental state, $\mbf x \in L_2$.
In all cases, we presume that the PIE is obtained from a classically well-posed PDE using the definitions and formulae in~\cite{shivakumar_TAC2024} (or Lemma~\ref{lem:gohberg_polynomial}). This ensures that any solution of the PIE, $\mbf x(t)$ defines a solution of the PDE as $\mbf u(t)=\mcl T \mbf x(t)$ and hence when a stability condition involves $\mcl T \mbf x$, this is interpreted as the state of the underlying PDE from whence the PIE was obtained.}

\subsection{Notions of Lyapunov Stability}\vspace{-1mm}
Let us begin with the notions of Lyapunov stability. These notions are particularly important for energy-conserving systems such as wave and beam equations.
\begin{defn}\label{def:PIE2PDE_lyapunov}
We say $\mcl T \dot{\mbf x}=\mcl A \mbf x$ is
\begin{enumerate}
\item \textbf{Lyapunov PIE to PDE stable} if $\norm{\mcl T\mbf x(t)}\le C \norm{\mbf x(0)}$
\item \textbf{Lyapunov PIE stable} if $\norm{\mbf x(t)}\le C \norm{\mbf x(0)}$
\item \textbf{Lyapunov PDE stable} if $\norm{\mcl T\mbf x(t)}\le C \norm{\mcl T\mbf x(0)}$
\item \textbf{Lyapunov PDE to PIE stable} if $\norm{\mbf x(t)}\le C \norm{\mcl T \mbf x(0)}$ 
    \end{enumerate}
for some $C>0$ and any $\mbf x(t)$ with $\mcl T \dot{\mbf x}(t)=\mcl A {\mbf x}(t)$.
\end{defn}
Note that PDE to PIE stable never holds when the PIE state is a partial derivative of the PDE state ($\mbf x=D^{\alpha}\mbf u$) since at time $t=0$, this requires $\norm{\mbf x(0)}\le C \norm{\mcl T\mbf x(0)}$. However, such notions may hold for integro-differential equations~\cite{appell2000partial}.
%
In the following, we examine the relationships between the several proposed notions of Lyapunov stability.

\begin{lem}\label{lem:Lyapunov stability_relations} Given $\mcl T \dot{\mbf x}=\mcl A \mbf x$, the following are true.
\begin{enumerate}
\item PIE stable implies PIE to PDE stable
\item PDE stable implies PIE to PDE stable
\item PDE to PIE stable implies PIE stable, PDE stable and  PIE to PDE stable
\end{enumerate}
\end{lem}
\begin{proof}
For statement 1), we have $\norm{\mbf x(t)}\le C \norm{\mbf x(0)}$ implies PIE to PDE stable since
$
\norm{\mcl T\mbf x(t)}\le \norm{\mcl T}\norm{\mbf x(t)}\le \norm{\mcl T} C \norm{\mbf x(0)}.$
For statement 2), we have $\norm{\mcl T\mbf x(t)}\le C \norm{\mcl T\mbf x(0)}$ implies PIE to PDE stable since $
\norm{\mcl T\mbf x(t)}\le C \norm{\mcl T\mbf x(0)}\le  \norm{\mcl T} C \norm{\mbf x(0)}.$ For statement 3), we have that $\norm{\mbf x(t)}\le C \norm{\mcl T\mbf x(0)}$ implies PDE stable since $\norm{\mcl T\blue{\mbf x}(t)}\le \norm{\mcl T}\norm{\mbf x(t)}\le \norm{\mcl T} C \norm{\mcl T\mbf x(0)}$.
Similarly, PIE stable follows since $\norm{\mbf x(t)}\le C \norm{\mcl T\mbf x(0)}\le  \norm{\mcl T} C \norm{\mbf x(0)}$. Statements 1) or 2) then imply PIE to PDE stable.
\end{proof}

Lemma~\ref{lem:Lyapunov stability_relations} shows that PIE to PDE stable is the weakest notion, since it is implied by any of the others. However, PDE stable does not imply PIE stable and PIE stable does not imply PDE stable.
\paragraph{Conditions for Lyapunov Stability}
Having defined several notions of Lyapunov stability, we now provide conditions for such stability properties in terms of properties of a candidate Lyapunov function.

\begin{lem}[Lyapunov stability]\label{lem:PIE2PDE_lyapunov}
  Suppose $V: L_2 \rightarrow \R^+$ with $V(0)=0$ and $\dot V$ as in Defn.~\ref{def:Lyap_negative}. Then $\mcl T \dot{\mbf x}=\mcl A \mbf x$ is
\begin{enumerate}
\item PIE to PDE stable if $V$ is PDE positive, PIE bounded, and $\dot V$ negative semidefinite.
\item PIE stable if $V$ is PIE positive, PIE bounded, and $\dot V$ negative semidefinite.
\item PDE stable if $V$ is PDE positive, PDE bounded, and $\dot V$ negative semidefinite.
\item PDE to PIE stable if $V$ is PIE positive, PDE bounded, and $\dot V$ negative semidefinite.
\end{enumerate}
\end{lem}
\begin{proof}
In all cases, $\dot V\le 0$ implies $V(\mbf x(t))\le V(\mbf x(0))$. Hence, for statement 1, we have\vspace{-2mm}
\[
\epsilon \norm{\mcl T\mbf x(t)}^2\le V(\mbf x(t))\le V(\mbf x(0))\le C \norm{\mbf x(0)}^2.\vspace{-1mm}
\]
For statement 2,\vspace{-2mm}
\[
\epsilon \norm{\mbf x(t)}^2\le V(\mbf x(t))\le V(\mbf x(0))\le C \norm{\mbf x(0)}^2.\vspace{-1mm}
\]
For Statement 3,\vspace{-2mm}
\[
\epsilon\norm{\mcl T\mbf x(t)}^2\le V(\mbf x(t))\le V(\mbf x(0))\le C \norm{\mcl T\mbf x(0)}^2.\vspace{-1mm}
\]
For statement 4,\vspace{-2mm}
\[
\epsilon \norm{\mbf x(t)}^2\le V(\mbf x(t))\le V(\mbf x(0))\le C \norm{\mcl T\mbf x(0)}^2.\vspace{-1mm}\vspace{-2mm}
\]
\end{proof}\vspace{-2mm}
\vspace{-2mm}
%
%
%
%

\subsection{Notions of Exponential Stability}\vspace{-1mm}
For exponential stability, from weak to strong, we have
\begin{defn}\label{def:PIE2PDE_exponential}
We say $\mcl T \dot{\mbf x}=\mcl A \mbf x$ is
\begin{enumerate}
\item \textbf{Exp. PIE to PDE stable} if $\norm{\mcl T\mbf x(t)}\le C e^{-\alpha t}\norm{\mbf x(0)}$
\item \textbf{Exp. PIE stable} if $\norm{\mbf x(t)}\le C e^{-\alpha t}\norm{\mbf x(0)}$
\item \textbf{Exp. PDE stable} if $\norm{\mcl T\mbf x(t)}\le C e^{-\alpha t}\norm{\mcl T\mbf x(0)}$
\item \textbf{Exp. PDE to PIE stable} if $\norm{\mbf x(t)}\le C e^{-\alpha t}\norm{\mcl T\mbf x(0)}$
    \end{enumerate}
for some $C,\alpha>0$ and any $\mbf x(t)$ with $\mcl T \dot{\mbf x}(t)=\mcl A {\mbf x}(t)$.
\end{defn}
As in Lemma~\ref{lem:Lyapunov stability_relations}, Exp. PIE or PDE stable implies Exp. PIE to PDE stable and Exp. PDE to PIE stable implies all others.

\noindent \textbf{Lyapunov Conditions for Exponential Stability}
To keep the conditions tight and concise, in Lemma~\ref{lem:PIE2PDE_exponential} we use Lyapunov negativity instead of PIE or PDE negativity. This allows us to directly use $V(\mbf x(t))\le -\alpha V(\mbf x(t))$ to conclude $V(\mbf x(t))\le e^{-\alpha t}V(\mbf x(0))$. We will return to PIE and PDE negativity in Subsec.~\ref{subsec:FE_stability} and Sec.~\ref{sec:LPI_conditions}, while for now recalling only that the relation to Lyapunov negativity depends on PIE and PDE positivity and boundedness. 

\begin{lem}[Conditions for Exponential stability]\label{lem:PIE2PDE_exponential}
  Suppose $V: L_2 \rightarrow \R^+$ with $V(0)=0$ and $\dot V$ as in Defn.~\ref{def:Lyap_negative} is Lyapunov negative. Then $\mcl T \dot{\mbf x}=\mcl A \mbf x$ is
\begin{enumerate}
\item Exp. PIE to PDE stable if $V$ is PDE positive, PIE bounded.
\item Exp. PIE stable if $V$ is PIE positive, PIE bounded.
\item Exp. PDE stable if $V$ is PDE positive, PDE bounded.
\item Exp. PDE to PIE stable if $V$ is PIE positive, PDE bounded.
\end{enumerate}
\end{lem}
\begin{proof}
For statement 1, we have\vspace{-2mm}
\[
\epsilon \norm{\mcl T\mbf x(t)}^2\le V(\mbf x(t))\le e^{-\alpha t} V(\mbf x(0))\le C e^{-\alpha t} \norm{\mbf x(0)}^2.\vspace{-1mm}
\]
For statement 2,\vspace{-2mm}
\[
\epsilon \norm{\mbf x(t)}^2\le V(\mbf x(t))\le e^{-\alpha t} V(\mbf x(0))\le C e^{-\alpha t} \norm{\mbf x(0)}^2.\vspace{-1mm}
\]
For statement 3,\vspace{-2mm}
\[
\epsilon \norm{\mcl T\mbf x(t)}^2\le V(\mbf x(t))\le e^{-\alpha t} V(\mbf x(0))\le C e^{-\alpha t} \norm{\mcl T\mbf x(0)}^2.\vspace{-1mm}
\]
For statement 4,\vspace{-2mm}
\[
\epsilon \norm{\mbf x(t)}^2\le V(\mbf x(t))\le e^{-\alpha t} V(\mbf x(0))\le C e^{-\alpha t} \norm{\mcl T\mbf x(0)}^2.\vspace{-1mm}
\]
\end{proof}\vspace{-2mm}
\vspace{-1mm}\vspace{-1mm}

\subsection{More Exotic Flavours: Finite Energy Stability}\label{subsec:FE_stability}\vspace{-1mm}
One might interpret ``non-coercive'' Lyapunov functions as $V$ positive definite, but not PDE positive. However, we don't have a consistent mechanism to enforce $V$ positive definite and so we will relax this to $V$ positive semidefinite. In order to ensure a non-zero Lyapunov function, we then enforce a strict negativity condition on the derivative. A consequence of this formulation is that the Lyapunov function does not enforce a pointwise-in-time decay rate in the same way that exponential stability does, but rather only imposes a bound on the energy of the solution as measured by its \textit{temporal} $L_2$ norm -- i.e. $\norm{\mbf x}_{L_2} = \sqrt{\int_0^\infty \norm{\mbf x(t)}^2 dt}$. We refer to this notion of energy as "Finite-Energy" Stability\footnote{See Note~\ref{note:asymptotic} in~\cite{peet_2026LCSS_arxiv} to relate asymptotic and finite-energy stability.}.  

\begin{defn}\label{def:PIE2PDE_FE_stable}
We say $\mcl T \dot{\mbf x}=\mcl A \mbf x$ is
\begin{enumerate}
\item \textbf{FE PIE to PDE stable } if $\norm{\mcl T\mbf x}_{L_2}\le C\norm{\mbf x(0)}$
\item \textbf{FE PIE stable} if $\norm{\mbf x}_{L_2}\le C\norm{\mbf x(0)}$
\item \textbf{FE PDE stable} if $\norm{\mcl T\mbf x}_{L_2}\le C\norm{\mcl T\mbf x(0)}$
\item \textbf{FE PDE to PIE stable} if $\norm{\mbf x}_{L_2}\le C\norm{\mcl T\mbf x(0)}$
    \end{enumerate}
    for some $C>0$ and any $\mbf x(t)$ which satisfies $\mcl T \dot{\mbf x}(t)=\mcl A {\mbf x}(t)$.
\end{defn}

%
%
%

\noindent \textbf{Lyapunov Conditions for FE stability} Unlike exponential stability, negativity of the derivative is critical in FE stability. Hence, we do not use Lyapunov negativity in Lemma~\ref{lem:FE_stability}.

\begin{lem}[Conditions for Finite Energy Stability]\label{lem:FE_stability}
  Suppose $V: L_2 \rightarrow \R^+$ with $V(0)=0$ is positive semidefinite and $\dot V$ is as in Defn.~\ref{def:Lyap_negative}. Then $\mcl T \dot{\mbf x}=\mcl A \mbf x$ is
  \begin{itemize}
\item FE PIE to PDE stable if $V$ is PIE bounded, $\dot V$ PDE negative.
\item FE PIE stable if $V$ is PIE bounded, $\dot V$ PIE negative.
\item FE PDE stable if $V$ is PDE bounded, $\dot V$ PDE negative.
\item FE PDE to PIE stable if $V$ is PDE bounded, $\dot V$ PIE negative.
  \end{itemize}
\end{lem}
\begin{proof}
In each case, we use $V(\mbf x)\ge 0$ and integrate the Lyapunov negativity condition. Specifically, for the PDE negativity property, we integrate $\dot V(\mbf x(t))\le -\alpha \norm{\mcl T \mbf x(t)}^2$ to show that for any $T$, $V(\mbf x(T))\ge 0$ and hence\vspace{-2mm}
\begin{align*}
&V(\mbf x(T))- V(\mbf x(0))\le -\alpha \int_0^T  \norm{\mcl T\mbf x(t)}^2 dt\qquad \\[-2mm]
&\Rightarrow \qquad \int_0^T  \norm{\mcl T \mbf x(t)}^2 dt \le \frac{1}{\alpha}V(\mbf x(0)).\vspace{-1mm}
\end{align*}
We may then take the limit as $T\rightarrow \infty$ to show that\vspace{-1mm}
\[
\alpha \norm{\mcl T \mbf x}_{L_2}^2 \le V(\mbf x(0)).\vspace{-1mm}
\]
For PIE negativity, we similarly have $\alpha \norm{\mbf x}_{L_2}^2 \le V(\mbf x(0))$. Now, statement 1) follows since 
$\alpha \norm{\mcl T\mbf x}_{L_2}^2 \le V(\mbf x(0))\le C \norm{\mbf x(0)}^2$. 

Statement 2) from $\alpha \norm{\mbf x}_{L_2}^2 \le V(\mbf x(0))\le C \norm{\mbf x(0)}^2$. 

Statement 3) from $\alpha \norm{\mcl T\mbf x}_{L_2}^2 \le V(\mbf x(0))\le C \norm{\mcl T\mbf x(0)}^2$. 

Statement 4) from  $\alpha \norm{\mbf x}_{L_2}^2 \le V(\mbf x(0))\le C \norm{\mcl T\mbf x(0)}^2$.
\end{proof}
%
%
\section{Sufficient Conditions for Stability}\label{sec:LPI_conditions}\vspace{-1mm}
In this section, we briefly provide a list of conditions for our stability notions in terms of feasibility of operator inequalities defined in terms of $\mcl T$ and $\mcl A$ and variable $\mcl P \in \Pi_3$ or $\mcl Q \in \Pi_3$. Here the variable label $\mcl P$ will be used when the variable is constrained to be self-adjoint and $\mcl Q$ otherwise. The operator inequality constraints are all interpreted as positivity on $L_2$ and may be verified by semidefinite programming algorithms in\footnote{Alternatively, these may be tested via discretization -- i.e. late lumping.} software such as PIETOOLS~\cite{PIETOOLS2025}.  


\begin{lem}\label{lem:LPI_bounding}
Given $\mcl T,\mcl P\in \Pi_3$, if $V(\mbf x)=\ip{\mbf x}{\mcl P \mbf x}$, then 
\begin{enumerate}
\item positive semidefinite is equivalent to  $\mcl P \ge 0$
\item PDE positivity is equivalent to  $\mcl P \ge \epsilon \mcl T^* \mcl T$
\item PIE positivity is equivalent to  $\mcl P \ge \epsilon I$
\item PDE bounded is equivalent to  $\mcl P \le C \mcl T^* \mcl T$
\item PIE bounded is equivalent to  $\mcl P \le C I$
\end{enumerate}
for some $\epsilon,C>0$. If $\dot V(\mbf x)=\ip{\mbf x}{\mcl D \mbf x}$ (per Defn.~\ref{def:Lyap_negative}), then
\begin{enumerate}
\item PDE negative is equivalent to $\mcl D \le -\alpha \mcl T^* \mcl T$
\item  PIE negative is equivalent to $\mcl D \le -\alpha I$
\item  Lyapunov negative is equivalent to $\mcl D \le -\alpha P$
\end{enumerate}
for some $\alpha>0$. See Note~\ref{note:lemma6_proof} in~\cite{peet_2026LCSS_arxiv} for a proof.
\end{lem}
Note that PIE bounded in satisfied for any $\mcl P \in \Pi_3$.

Every sufficient condition proposed in this section relies on one of the following two canonical candidate Lyapunov functions. Specifically, given $\mcl T$, we have \vspace{-2mm}
\[V_1(\mbf x):=\ip{\mcl T\mbf x}{ \mcl P \mcl T \mbf x}\quad \text{and}\quad  V_2(\mbf x):=\ip{\mbf x}{\mcl Q \mcl T \mbf x}\vspace{-2mm}
\]
where $V_1$ and $V_2$ are parameterized by PI operators $\mcl P\in \Pi_3$ and $\mcl Q\in \Pi_3$, respectively. Here we constrain $\mcl P=\mcl P^*$ and $\mcl Q \mcl T=\mcl T^*\mcl Q^*$. See Notes~\ref{note:V1} and~\ref{note:V2} in~\cite{peet_2026LCSS_arxiv} for further justification of $V_1$ and $V_2$, respectively.\vspace{-1mm}

\subsection{Operator Inequalities for Lyapunov Stability}\vspace{-1mm}
In the following, we combine $V_1$ and $V_2$ with Lemmas~\ref{lem:PIE2PDE_lyapunov} and~\ref{lem:LPI_bounding} to obtain operator inequalities for each category of Lyapunov stability -- See Note~\ref{note:lemma6_proof} in~\cite{peet_2026LCSS_arxiv} for a proof.

\begin{lem}\label{lem:PIE2PDE_lyapunov_LPI}
Given $\mcl T,\mcl A$, the PIE $\mcl T \dot{\mbf x}=\mcl A \mbf x$ is
\begin{enumerate}
\item PIE to PDE stable if $\mcl Q \mcl T=(\mcl Q \mcl T)^*\ge \epsilon \mcl T^* \mcl T$ and $\mcl Q\mcl A+\mcl A^*\mcl Q^*\le 0$.
\item PIE stable if $\mcl Q \mcl T=(\mcl Q \mcl T)^* \ge \epsilon I$  and $\mcl Q\mcl A+\mcl A^*\mcl Q^*\le 0$.
\item PDE stable if $\mcl P \ge \epsilon I$, $\mcl T^* \mcl P \mcl A+\mcl A^* \mcl P \mcl T\le 0$.
\item PDE to PIE stable if $\mcl Q \mcl T=(\mcl Q \mcl T)^* \ge \epsilon I$, $\mcl Q \mcl T\le C \mcl T^* \mcl T$, and $\mcl Q\mcl A+\mcl A^*\mcl Q^*\le 0$.
\end{enumerate}
for some $\mcl Q,\mcl P \in \Pi_3$, $\epsilon,C>0$. See Note~\ref{note:lemma7_proof} in~\cite{peet_2026LCSS_arxiv} for proof.
\end{lem}


\begin{example}[PIE to PDE Stability of the Wave Equation]\label{ex:wave_stable}
Let us now reconsider the wave equation which we recall is not PDE stable. Specifically, the wave equation $\ddot{\mbf u}=\mbf u_{ss}$ with $\mbf u(t,0)=\mbf u(t,1)=0$ admits a PIE representation of the form\vspace{-2mm}
\begin{equation}
\overbrace{\bmat{I&0\\0&\mcl T_0}}^{\mcl T}\bmat{\dot{\mbf x}_1\\ \dot{\mbf x}_2}=\overbrace{\bmat{0&I\\ I& 0}}^{\mcl A}\bmat{{\mbf x_1}\\ {\mbf x_2}}.
\end{equation}
where $\mbf x_1:=\dot{\mbf u}$ and $\mbf x_2:=\mbf{u}_{ss}$ and where\vspace{-2mm}
\[
(\mcl T_0\mbf x)(s):=\int_0^1 G_T(s,\theta)\mbf x(\theta)d\theta
\quad G_T(s,\theta)=
\begin{cases}
-\theta&\theta < s\\
-s& s < \theta.
\end{cases}\vspace{-2mm}
\]
To apply the stability conditions of Lemma~\ref{lem:PIE2PDE_lyapunov}, we observe that $\mcl T_0=-\mcl R^* \mcl R$ for\vspace{-2mm}
\[
(\mcl R\mbf x)(s):=\int_0^1 G_R(s,\theta)\mbf x(\theta)d\theta
\quad G_R(s,\theta)=
\begin{cases}
0&\theta < s\\
-1& s < \theta.
\end{cases}\vspace{-2mm}
\]
Thus if we choose $\mcl Q={\tiny \bmat{I &0\\0&-I}}$, we find that \vspace{-1mm}
\begin{equation}
\mcl T_0^* \mcl T_0=\mcl R^* \mcl R\mcl R^* \mcl R\le \norm{\mcl R \mcl R^*}\mcl R^* \mcl R=-\norm{\mcl T_0}\mcl T_0\vspace{-2mm}
\end{equation}
and\vspace{-2mm}
\begin{align}
\mcl Q \mcl T&=\mcl T^* \mcl Q^*=\bmat{I&0\\0&-\mcl T_0}\\
&\ge \min\left\{1,\frac{1}{\norm{\mcl T_0}}\right\}\bmat{I&0\\0&\mcl T_0^*\mcl T_0}=\min\left\{1,\frac{1}{\norm{\mcl T_0}}\right\}\mcl T^* \mcl T.\notag
\end{align}
Furthermore,\vspace{-2mm}
\[
\mcl Q\mcl A+\mcl A^*\mcl Q^*=\bmat{I &0\\0&-I}\bmat{0&I\\ I& 0}+\bmat{0&I\\ I& 0}\bmat{I &0\\0&-I}=0.\vspace{-2mm}
\]
Hence Lemma~\ref{lem:PIE2PDE_lyapunov_LPI} implies PIE to PDE stability.\vspace{-2mm}
\end{example}

\subsection{Operator Inequalities for Exponential Stability}\vspace{-1mm}
We combine $V_1$ and $V_2$ with Lemmas~\ref{lem:PIE2PDE_exponential} and~\ref{lem:LPI_bounding} to obtain the following operator inequalities for each category of exponential stability\footnote{If we wish to avoid bilinearity in decision variables $\alpha$ and $\mcl Q$ or $\mcl P$, we may replace the Lyapunov negativity conditions in Lem.~\ref{lem:PIE2PDE_exponential} with PIE or PDE negativity -- See Note~\ref{note:exp_LPI_alt}.}.

\begin{lem}[LPIs for Exp. Stability]\label{lem:exp_stable_LPIs}
Given $\mcl T,\mcl A$, the PIE $\mcl T \dot{\mbf x}=\mcl A \mbf x$ is
\begin{enumerate}
\item Exp. PIE to PDE stable if $\mcl Q \mcl T=(\mcl Q \mcl T)^* \ge \epsilon \mcl T^* \mcl T$, $\mcl Q\mcl A+\mcl A^*\mcl Q^*\le -\alpha \mcl Q \mcl T$
\item Exp. PIE stable if $\mcl Q \mcl T=(\mcl Q \mcl T)^* \ge \epsilon I$  and $\mcl Q\mcl A+\mcl A^*\mcl Q^*\le -\alpha \mcl Q \mcl T$
\item Exp. PDE stable if $\mcl P \ge \epsilon I$, $\mcl T^* \mcl P \mcl A+\mcl A^* \mcl P \mcl T\le -\alpha\mcl T^* \mcl P \mcl T$
\item Exp. PDE to PIE stable if $\mcl Q \mcl T=(\mcl Q \mcl T)^* \ge \epsilon I$, $\mcl Q \mcl T\le C \mcl T^* \mcl T$, and $\mcl Q\mcl A+\mcl A^*\mcl Q^*\le -\alpha \mcl Q \mcl T$
  \end{enumerate}
  for some $\mcl Q,\mcl P \in \Pi_3$ and $\epsilon,C,\alpha>0$. See Note~\ref{note:lemma8_proof} in~\cite{peet_2026LCSS_arxiv} for proof.
\end{lem}

\begin{example}[The Heat Equation is Exp. PDE stable]
The heat equation $\mbf u_t=\mbf u_{ss}+\lambda \mbf u$ with Dirichlet boundary conditions admits a PIE representation $\mcl T \dot{\mbf x}=\mcl A \mbf x=(I+\lambda \mcl T)\mbf x$ where \vspace{-1mm}
\[
(\mcl T\mbf x)(s)=\int_0^s  \theta (s-1)\mbf x(\theta)  d\theta + \int_s^1  s (\theta-1)\mbf x(\theta)  d\theta\vspace{-1mm}
\]
and $\mcl T=\mcl T^*=-\mcl M^*\mcl M$ where \vspace{-1mm}
\[
(\mcl M\mbf x)(s)=\int_0^s  \theta \mbf x(\theta)  d\theta + \int_s^1  (\theta-1)\mbf x(\theta)  d\theta.\vspace{-1mm}
\]
Now, $-\mcl M^*\mcl M\le -\pi^2 \mcl T^* \mcl T$. Hence for $\mcl P=I$, we have that\vspace{-1mm}
\[
\mcl T^* \mcl P \mcl A+\mcl A^* \mcl P \mcl T=-2(\mcl M^*\mcl M-\lambda \mcl T^* \mcl T)\le -2(\pi^2-\lambda)\mcl T^* \mcl T.\vspace{-1mm}
\]
Hence by Lemma~\ref{lem:exp_stable_LPIs} we have that the heat equation is exponentially PDE stable for any $\lambda<\pi^2$.\vspace{-0mm}

\end{example}
\subsection{Operator Inequalities  for Finite Energy Stability}\vspace{-1mm}
Finally, we combine $V_1$ and $V_2$ with Lemmas~\ref{lem:FE_stability} and~\ref{lem:LPI_bounding} to obtain the following operator inequalities for each category of finite energy stability.

\begin{lem}\label{lem:FE_stability_LPI}
Given $\mcl T,\mcl A$, the PIE $\mcl T \dot{\mbf x}=\mcl A \mbf x$ is
  \begin{enumerate}
\item FE PIE to PDE stable if $\mcl Q \mcl T=(\mcl Q \mcl T)^* \ge 0$, $\mcl Q\mcl A+\mcl A^*\mcl Q^*\le -\alpha \mcl T^* \mcl T$.
\item FE PIE stable if $\mcl Q \mcl T=(\mcl Q \mcl T)^* \ge 0$  and $\mcl Q\mcl A+\mcl A^*\mcl Q^*\le -\alpha I$.
\item FE PDE stable if $\mcl P \ge 0$, $\mcl T^* \mcl P \mcl A+\mcl A^* \mcl P \mcl T\le -\alpha \mcl T^* \mcl T$.
\item FE PDE to PIE stable if $\mcl Q \mcl T=(\mcl Q \mcl T)^* \ge 0$, $\mcl Q \mcl T\le C \mcl T^* \mcl T$, and $\mcl Q\mcl A+\mcl A^*\mcl Q^*\le -\alpha I$.
  \end{enumerate}
  for some $\mcl Q,\mcl P \in \Pi_3$ and $\epsilon,C,\alpha>0$. See Note~\ref{note:lemma10_proof} in~\cite{peet_2026LCSS_arxiv} for proof.
\end{lem}\vspace{-1mm}
\section{Summary}\label{sec:summary}\vspace{-1mm}
Table~\ref{tab:summary} summarizes the weakest Lyapunov conditions required for each notion of stability.
\begin{table}[hb]\vspace{-4mm}
\begin{center}{\scriptsize
\begin{tabular}{c||c|c||c|c||c|c|c}
{\hspace{-2mm}\scriptsize  }\hspace{-2mm} &\multicolumn{2}{c}{Positivity} & \multicolumn{2}{c}{Upper Bound}& \multicolumn{3}{c}{Negativity} \\
{\hspace{-2mm}\scriptsize  Notion }\hspace{-2mm} 
                & PIE  & PDE & PDE & PIE & PIE & PDE& SD\\
\hline
PIE2PDE         &    & X &   & X &  &   & X\\
 PIE            & X  &   &   & X &  &   & X\\
 PDE            &    & X & X &   &  &   & X\\
 PDE2PIE        & X  &   & X &   &  &   & X\\
Exp PIE2PDE    &    & X &   & X & X &    & \\
Exp PIE2PDE    &    & X & X &   &   &  X & \\
Exp  PIE       & X  &   &   & X & X &    & \\
Exp  PIE       & X  &   &  X&   &   &  X  & \\
Exp  PDE       &    & X & X &   &   & X  &  \\
Exp  PDE2PIE   & X  &   & X &   &   & X  & \\
FE PIE2PDE     &    &   &   & X &   & X   &  \\
FE PIE         &    &   &   & X & X &    & \\
FE PDE         &    &   & X &   &   & X  & \\
FE PDE2PIE     &    &   & X &   & X &    & \\
\end{tabular}}
\end{center}\vspace{-2mm}
\caption{The weakest Lyapunov conditions which imply each given notion of stability (2 variants for PIE2PDE and PIE exp. stability). Categories of positivity, bound, and negativity are ordered from strongest to weakest. }\label{tab:summary}\end{table}\vspace{-2mm}\vspace{-2mm}

\section{PDEs and Stability Analysis in PIETOOLS}\label{sec:summary}\vspace{-1mm}
\blue{We consider 4 illustrative PDEs and test each notion of stability indicated in Table~\ref{tab:summary} using PIETOOLS to enforce the associated LPI constraints as indicated in Lemmas~\ref{lem:PIE2PDE_lyapunov_LPI},~\ref{lem:exp_stable_LPIs}, and~\ref{lem:FE_stability_LPI}. In all cases, the PDE, BCs, and definition of PIE state, $\mbf x$, are given. LPI parameters were chosen as $\alpha=\epsilon=.001$. Table~\ref{tab:results} lists each notion of stability PIETOOLS was able to verify for each example\footnote{Note that a lack of verification does not necessarily imply the system is not stable in the given sense.}.} 

\noindent \textbf{Example 1:}  Hyperbolic System of Balance Laws~\cite{diagne2012lyapunov}\vspace{-2mm}
\[
\dot{\mbf u}=-M\mbf u-L\mbf u_s\qquad K_1\mbf u(t,0)+K_2 \mbf u(t,1)=0 \qquad \mbf x:=\mbf u_s \vspace{-4mm}
\]
{\tiny
\begin{align*}
L&=\bmat{   -3.76 \hspace{-2mm}  &     0  \hspace{-2mm} &      0\\
         0\hspace{-2mm}  &  0.5 \hspace{-2mm} &       0\\
         0 \hspace{-2mm} &       0 \hspace{-2mm}  & 5.25}\quad M=\bmat{    1.03 \hspace{-2mm} & .05\hspace{-2mm} &  .92\\
   1.03\hspace{-2mm} &  .05\hspace{-2mm} &  .92\\   1.03\hspace{-2mm} &  .05 \hspace{-2mm} & .92}\\
K_1&=\bmat{    1\hspace{-2mm} &  -1.95 \hspace{-2mm} & 34.5\\   0   \hspace{-2mm} &     0  \hspace{-2mm}  &     0\\       0 \hspace{-2mm}  &      0  \hspace{-2mm} &      0},\quad 
K_2=\bmat{          0   \hspace{-2mm}  &    0  \hspace{-2mm}  &     0\\   -0.55 \hspace{-2mm} &  1   \hspace{-2mm}  &    0\\   -1.75 \hspace{-2mm} &       0  \hspace{-2mm}&  1}\\[-6mm]
\end{align*}}
\noindent \textbf{Example 2:} Reaction-Diffusion w/ spatial coefficients~\cite{gahlawat_2017TAC}.\vspace{-2mm}
\begin{align*}
\dot{\mbf u}(t,s)&=a(s)\mbf u(t,s)+b(s)\mbf u_s(t,s)+c(s)\mbf u(s)\\
\mbf u(t,0)&=\mbf u_s(t,1)=0 \qquad \mbf x:=\mbf u_{ss}\qquad a(s)=s^3-s^2+2\\
b(s)&=3s^2 - 2s,\quad c(s)=-s^3/2 + 1.3s^2- 1.5s + 5.35\\[-7mm]
\end{align*}

%
\noindent \textbf{Example 3:}  Euler-Bernoulli Beam Equation. Given $\ddot{\mbf u}=-\mbf u_{ssss}$, use the first-order form defined by $\mbf u_1=\dot{\mbf u}$ and $\mbf u_2=\mbf u$, and we have $\dot{\mbf u}_1=-\mbf u_{2,ssss}$, $\dot{\mbf u}_2=\mbf u_1$
where $\mbf u_2(t,0)=\mbf u_{2,ss}(t,1)=\mbf u_{2,s}(t,0)=\mbf u_{2,sss}(t,1)=0$. Fundamental state is defined as $\mbf x:=\bmat{\dot{\mbf u}& \mbf u_{ssss}}^T$.\vspace{-0mm}

\noindent \textbf{Example 4:}  Tip-Damped Wave Equation~\cite{chen1979energy}. Given $\ddot{\mbf u}=\mbf u_{ss}-2 \dot{\mbf u}-\mbf u$, use the first-order form defined by $\mbf u_1=\mbf u_t$ and $\mbf u_2=\mbf u$, and we have
\[
\dot{\mbf u}=\bmat{-2 & -1\\1&0}\mbf u+\bmat{1&0\\0&0}\mbf u_{ss} \qquad \mbf x:=\bmat{\dot{\mbf u}_s\\ \mbf u_{ss}}
\]
where $\mbf u_1(t,0)=\mbf u_2(t,0)=0$ and $\mbf u_s(t,1)=-\dot{\mbf u}(t,1)$.

\begin{table}[hb]\vspace{-4mm}
\begin{center}{\scriptsize
\begin{tabular}{c||c|c|c||c|c|c||c|c}
{\hspace{-2mm}\scriptsize  }\hspace{-2mm} &\multicolumn{3}{c}{Lyapunov} & \multicolumn{3}{c}{Exponential}& \multicolumn{2}{c}{Finite Energy} \\
{\hspace{-2mm}\scriptsize  Ex. }\hspace{-2mm} 
                & \hspace{-.5mm}PIE2PDE\hspace{-.5mm}  & \hspace{-.5mm}PIE\hspace{-.5mm} & PDE & \hspace{-1.5mm} PIE2PDE \hspace{-1.5mm} & \hspace{-.5mm}PIE\hspace{-.5mm} & PDE&  \hspace{-.5mm}PIE2PDE\hspace{-.5mm}  & PDE \\
\hline
1         &  X  &  & X  & X &  &  X & X&X\\
2         &  X  &  & X  & X &  &  X & X&X\\
3           &  X  &  &  &   &  &   & &\\
4        & X  &   &  &   &  &   & X &\\
\end{tabular}}
\end{center}\vspace{-2mm}
\caption{Notions of stability which hold for Examples 1-4 using the indicated definition of $\mbf u$ and $\mbf x$ as verified by PIETOOLS~\cite{PIETOOLS2025}.}\label{tab:results}\end{table}\vspace{-2mm}\vspace{-2mm}

\section{Conclusion}\vspace{-2mm}
We have proposed a hierarchical classification of Lyapunov, exponential, and finite-energy stability notions in terms of the fundamental state, $\mbf x$, which is related to the PDE state as $\mbf x=D^\alpha \mbf u$ and $\mbf u=\mcl T \mbf x$ and which satisfies $\mcl T \dot{\mbf x}=\mcl A \mbf x$ for some PI operators, $\mcl T,\mcl A$. This framework unifies Lyapunov stability conditions which may or may not include partial derivatives of the PDE state in lower bounds, upper bounds, and negativity conditions. Lyapunov conditions which establish stability, but not PDE stability often imply PIE to PDE stability. This framework allows for classification of PDEs such as the wave and beam equations. Furthermore, operator inequality conditions are provided for each notion of stability and may be verified using software such as PIETOOLS.\vspace{-2mm}

\vspace{-1mm}
\bibliographystyle{IEEEtran}
\bibliography{peet_bib,CDC_2026}

\appendix[Notes and Commentary]

\subsection{Notes and Commentary: Section~\ref{sec:LF_bounds}}\label{appx:notes}
\vspace{3mm}

\begin{note}\label{note:non_coercive}
To the list in Definition~\ref{def:Lyap_positive}, we might have added that $V$ is \textbf{Positive definite} if $V(\mbf x)> 0$ for all $x \neq 0$ -- i.e. positivity in the non-coercive sense. However, this definition is somewhat difficult to interpret constructively. One possible mechanism is to say that $V$ is positive definite if $V(\mbf x)\ge \epsilon \norm{\mcl R\mbf x}^2$ for some invertible mapping $\mcl R$ -- so that PDE positivity then becomes a special case with $\mcl R=\mcl T$. 

This is similar to the non-coercive Lyapunov function proposed for the heat equation Example 5.1 in~\cite{jacob_2018} which is defined as $V(\mbf u)=-\ip{A^{-1}\mbf u}{\mbf u}$ where $A\mbf u=\mbf u''$ where $A:\mcl D \rightarrow L_2$ with $\mcl D:=\{f \in H^2\;:\; f(1)=f(0)=0\}$.
Translating this into the notation used in the paper, we have $\mcl T := A^{-1}$ where
\[
(\mcl T \mbf x):=\int_0^s\theta (s-1)\mbf x(\theta)d\theta+\int_s^1 s (\theta-1)\mbf x(\theta)d\theta\vspace{-2mm}
\]
and $\mbf u=\mcl T\mbf x$ where $\mbf x=\mbf u''$. This is the same domain and hence the same $\mcl T=A^{-1}$ as was used in the illustrative example in the introduction. Hence the Lyapunov function $V(\mbf u)$ in~\cite{jacob_2018} may be expressed as $V(\mbf x)=-\ip{\mcl T\mbf x}{\mcl T \mcl T \mbf x}_{L_2}$. However, one can show in this case that $\mcl T=-\mcl M^* \mcl M$ where
\begin{align}
(\mcl M\mbf x)(s)&=\int_a^b G_M(s,\theta)\mbf x(\theta) e_n d\theta\label{eqn:operator_M} \\ \notag
G_M(s,\theta)&=
\begin{cases}
\theta&\theta < s\\
\theta-1 & s < \theta
\end{cases}
\end{align}
and hence the Lyapunov function has the form $V(\mbf x)=\ip{\mcl M\mcl T\mbf x}{\mcl M \mcl T \mbf x}_{L_2}=\norm{\mcl M \mcl T \mbf x}^2_{L_2}$. Note that, $\mcl M, \mcl T\in \Pi_2$, and since $\Pi_2$ is a composition *-algebra, $\mcl R:=\mcl M\mcl T$ is likewise an element of $\Pi_2$. Furthermore, because $\norm{\mcl R}\le \norm{\mcl M}\norm{\mcl T}$, we have that $V(\mbf x)=\norm{\mcl R \mbf x}^2_{L_2}$ is non-coercive since $V(\mbf x)\ge \epsilon \norm{\mcl T x}^2$ is equivalent to $-\mcl T=\mcl M^* \mcl M \ge \epsilon I$ which is not possible because $\mcl T,\mcl M \in \Pi_2$ are integral operators (There is not multiplier term -- i.e. $R_0=0$).
\end{note}

\vspace{3mm}

\subsection{Notes and Commentary: Section~\ref{sec:stability_notions}}

\vspace{3mm}

\begin{note}\label{note:asymptotic}
There is a significant question of whether FE stability implies asymptotic stability. Suppose we have FE PIE stable or FE PDE to PIE stable. Then $\mbf x \in L_2$ and hence
\[
\norm{\mcl T \dot{\mbf x}}_{L_2}=\norm{\mcl A \mbf x}_{L_2}\le \norm{\mcl A}_{\mcl L(L_2)}\norm{\mbf x}_{L_2}.
\]
Hence $\mcl T\dot{\mbf x} \in L_2$, which implies $\lim_{t\rightarrow \infty}\norm{\mcl T\mbf x(t)}\rightarrow 0$.  Unfortunately, however, this only holds for FE PIE stable and FE PDE to PIE stable. \blue{For an illustration of the use of FE stability to prove asymptotic stability, we refer to work in~\cite{yao_2025asymptotical}.} \vspace{-1mm}


%
\end{note}

\vspace{3mm}

\subsection{Notes and Commentary: Section~\ref{sec:LPI_conditions}}

\vspace{3mm}

\begin{note}[Proof of Lemma~\ref{lem:LPI_bounding}]\label{note:lemma6_proof}

Recall the statement of Lemma~\ref{lem:LPI_bounding}.

\begin{lem}[LPIs for Lyapunov Bounds]
Given $\mcl T,\mcl P \in \Pi_3$, if $V(\mbf x)=\ip{\mbf x}{\mcl P \mbf x}$, then 
\begin{enumerate}
\item positive semidefinite is equivalent to  $\mcl P \ge 0$
\item PDE positivity is equivalent to  $\mcl P \ge \epsilon \mcl T^* \mcl T$
\item PIE positivity is equivalent to  $\mcl P \ge \epsilon I$
\item PDE bounded is equivalent to  $\mcl P \le C \mcl T^* \mcl T$
\item PIE bounded is equivalent to  $\mcl P \le C I$
\end{enumerate}
for some $\epsilon,C>0$. If $\dot V(\mbf x)=\ip{\mbf x}{\mcl D \mbf x}$ with $\mcl D \in \Pi_3$, then
\begin{enumerate}
\item PDE negative is equivalent to $\mcl D \le -\alpha \mcl T^* \mcl T$
\item  PIE negative is equivalent to $\mcl D \le -\alpha I$
\item  Lyapunov negative is equivalent to $\mcl D \le -\alpha P$
\end{enumerate}
for some $\alpha>0$.
\end{lem}

\begin{proof}
Statement 1) follows from the definition of operator positivity $\mcl P\ge 0$. Statement 2) follows since $\mcl P - \epsilon \mcl T^* \mcl T\ge 0$ if and only if $V(\mbf x)-\epsilon \norm{\mcl T \mbf x}^2=\ip{\mbf x}{(\mcl P - \epsilon \mcl T^* \mcl T) \mbf x}\ge 0$ for all $\mbf x\in L_2$. Statement 3) follows since $\mcl P - \epsilon I\ge 0$ if and only if $V(\mbf x)-\epsilon \norm{\mbf x}^2=\ip{\mbf x}{(\mcl P - \epsilon I) \mbf x}\ge 0$ for all $\mbf x\in L_2$. Statement 4) follows since $C \mcl T^* \mcl T-\mcl P \ge 0$ if and only if $C\norm{\mcl T \mbf x}^2-V(\mbf x)=\ip{\mbf x}{(C \mcl T^* \mcl T-\mcl P) \mbf x}\ge 0$ for all $\mbf x\in L_2$. Statement 5) follows since $C I-\mcl P \ge 0$ if and only if $C\norm{\mbf x}^2-V(\mbf x)=\ip{\mbf x}{(C I-\mcl P) \mbf x}\ge 0$ for all $\mbf x\in L_2$.

Statement 6) follows since $\alpha \mcl T^* \mcl T-\mcl D \ge 0$ if and only if $\alpha \norm{\mcl T \mbf x}^2-\dot V(\mbf x)=\ip{\mbf x}{(\alpha \mcl T^* \mcl T-\mcl D) \mbf x}\ge 0$ for all $\mbf x\in L_2$. Statement 7) follows since $\alpha I-\mcl D \ge 0$ if and only if $\alpha \norm{\mbf x}^2-\dot V(\mbf x)=\ip{\mbf x}{(\alpha I-\mcl D) \mbf x}\ge 0$ for all $\mbf x\in L_2$.
\end{proof}

\end{note}

\vspace{3mm}

\begin{note}\label{note:V1}
\textbf{Properties of $V_1$:} First consider $V_1(\mbf x)$, which is perhaps the most obvious form of Lyapunov function, being defined solely in terms of the PDE state, $\mcl T \mbf x$. Applying Lem.~\ref{lem:LPI_bounding} with $\mcl P \mapsto \mcl T^* \mcl P \mcl T$, $V_1$ is PDE positive if $\mcl P \ge \epsilon I$. However, it is not usually possible for $V_1$ to be PIE positive, since this would imply $\mcl T^* \mcl P \mcl T \ge \epsilon I$. If we recall that typically $\mcl T \in \Pi_2$ (or at least some part thereof), this implies that $\mcl T^* \mcl P\mcl T \in \Pi_2$ and hence cannot be coercive. 
Next, we note that $V_1$ is both PDE and PIE bounded, since $\mcl T^* \mcl P \mcl T\le \norm{\mcl P}\mcl T^* \mcl T$ implies PDE bounded and PDE bounded implies PIE bounded. 
For $\mcl T \dot{\mbf x}=\mcl A\mbf x$, the derivative of $V_1$ is \vspace{-2mm}
\[
\dot V_1(\mbf x):=\ip{\mcl T\mbf x}{ \mcl P \mcl A \mbf x}+\ip{ \mcl P \mcl A \mbf x}{\mcl T\mbf x}=\ip{\mbf x}{(\mcl T^*\mcl P \mcl A+\mcl A^* \mcl P\mcl T)\mbf x}.\vspace{-2mm}
\]
Thus $\dot V_1$ is PDE negative if $\mcl T^* \mcl P \mcl A+\mcl A^* \mcl P \mcl T\le -\epsilon \mcl T^* \mcl T$. However, as for PIE positivity, $\dot V$ is unlikely to be PIE negative since this would require $\mcl T^* \mcl P \mcl A+\mcl A^* \mcl P \mcl T\le -\epsilon I$ which is not possible unless $\mcl T,\mcl A \in \Pi_3$.
\end{note}

\vspace{3mm}

\begin{note}\label{note:V2}
\textbf{Properties of $V_2$:}  Consider $V_2=\ip{\mbf x}{\mcl Q \mcl T \mbf x}$ where $\mcl Q \mcl T=(\mcl Q \mcl T)^*$. The difference between $V_1$ and $V_2$ is that $V_2$ is defined in terms of both the PDE state ($\mcl T \mbf x$) and PIE state ($\mbf x$). For example, in the heat equation, for $\mcl Q=-I$ we would have $V=-\ip{\mbf u}{\mbf u_{ss}}$ which reduces to something like $V=\ip{\mbf u_s}{\mbf u_s}$, depending on boundary conditions. 
Applying Lem.~\ref{lem:LPI_bounding} with $\mcl P \mapsto \mcl Q \mcl T$, $V_2$ is PDE positive if $\mcl Q \mcl T\ge \epsilon \mcl T^* \mcl T$. As with $V_1$, PIE positivity is not possible since it implies that $\mcl Q \mcl T \not\in \Pi_2$. $V_2$ is always PIE bounded and is PDE bounded if $\mcl Q \mcl T \le C \mcl T^* \mcl T$. 
For $\mcl T \dot{\mbf x}=\mcl A\mbf x$, the derivative of $V_2$ is \vspace{-2mm}
\[
\dot V_2(\mbf x):=\ip{\mbf x}{ \mcl Q \mcl A \mbf x}+\ip{\mcl Q \mcl A \mbf x}{ \mbf x}=\ip{\mbf x}{(\mcl Q \mcl A+\mcl A^*\mcl Q^*)\mbf x}. \vspace{-2mm}
\]
Thus $\dot V_2$ is PDE negative if $ \mcl Q \mcl A+\mcl A^* \mcl Q^*\le -\epsilon \mcl T^* \mcl T$ and PIE negative if $ \mcl Q \mcl A+\mcl A^* \mcl Q^*\le -\alpha I$. Unlike $\dot V_1$, it is possible for $\dot V_2$ to be PIE negative. Finally, note that $V_1$ is a special case of $V_2$ with $\mcl Q=\mcl T^*\mcl P$.
\end{note}

\vspace{3mm}

\begin{note}[Proof of Lemma~\ref{lem:PIE2PDE_lyapunov_LPI}]\label{note:lemma7_proof}

Recall the statement of Lemma~\ref{lem:PIE2PDE_lyapunov_LPI}.

\begin{lem}
Given $\mcl T,\mcl A$, the PIE $\mcl T \dot{\mbf x}=\mcl A \mbf x$ is
\begin{enumerate}
\item PIE to PDE stable if $\mcl Q \mcl T=(\mcl Q \mcl T)^*\ge \epsilon \mcl T^* \mcl T$ and $\mcl Q\mcl A+\mcl A^*\mcl Q^*\le 0$.
\item PIE stable if $\mcl Q \mcl T=(\mcl Q \mcl T)^* \ge \epsilon I$  and $\mcl Q\mcl A+\mcl A^*\mcl Q^*\le 0$.
\item PDE stable if $\mcl P \ge \epsilon I$, $\mcl T^* \mcl P \mcl A+\mcl A^* \mcl P \mcl T\le 0$.
\item PDE to PIE stable if $\mcl Q \mcl T=(\mcl Q \mcl T)^* \ge \epsilon I$, $\mcl Q \mcl T\le C \mcl T^* \mcl T$, and $\mcl Q\mcl A+\mcl A^*\mcl Q^*\le 0$.
\end{enumerate}
for some $\mcl Q,\mcl P \in \Pi_3$ and $\epsilon,C>0$. 
\end{lem}

\begin{proof}
For statement 1), let $V(\mbf x)=\ip{\mbf x}{\mcl Q\mcl T \mbf x}$. Then by Lem.~\ref{lem:LPI_bounding}, $\mcl Q \mcl T=(\mcl Q \mcl T)^* \ge \epsilon \mcl T^* \mcl T$ implies $V$ is PDE positive. Since $\mcl Q, \mcl T\in \Pi_3 \subset \mcl L(L_2)$, $\mcl Q \mcl T \le \norm{\mcl Q \mcl T}_{\mcl L(L_2)} I$ and hence by Lem.~\ref{lem:LPI_bounding} $V$ is PIE bounded.  If $\mcl T \dot{\mbf x}(t)=\mcl A \mbf x(t)$, we have
that $\mcl Q\mcl A+\mcl A^*\mcl Q^*\le 0$ implies
\begin{align*}
\frac{d}{dt}V(\mbf x(t))&=\ip{\dot{\mbf x}(t)}{\mcl Q\mcl T \mbf x(t)}+\ip{\mbf x(t)}{\mcl Q\mcl T \dot{\mbf x}(t)}\\
&=\ip{\mcl Q\mcl T\dot{\mbf x}(t)}{ \mbf x(t)}+\ip{\mbf x(t)}{\mcl Q\mcl A \mbf x(t)}\\
&=\ip{\mcl Q\mcl A\mbf x(t)}{ \mbf x(t)}+\ip{\mbf x(t)}{\mcl Q\mcl A \mbf x(t)}\le 0
\end{align*}
and hence $\dot V(x)$ is negative semidefinite. We conclude PIE to PDE stability by Lem.~\ref{lem:PIE2PDE_lyapunov}.

For statement 2), let $V(\mbf x)=\ip{\mbf x}{\mcl Q\mcl T \mbf x}$. Then by Lem.~\ref{lem:LPI_bounding}, $\mcl Q \mcl T=(\mcl Q \mcl T)^* \ge \epsilon I$ implies $V$ is PIE positive. As in the proof of statement 1), $V$ is PIE bounded and $\dot V$ negative semidefinite. We conclude PIE stability by Lem.~\ref{lem:PIE2PDE_lyapunov}.

For statement 3), let $V(\mbf x)=\ip{\mbf x}{\mcl T^* \mcl P \mcl T\mbf x}$. Then $\mcl P \ge \epsilon I$ implies $\mcl T^* \mcl P \mcl T \ge \epsilon \mcl T^*\mcl T$ and hence  by Lem.~\ref{lem:LPI_bounding}, $V$ is PDE positive. Since  $\mcl P\in \Pi_3 \subset \mcl L(L_2)$, $\mcl T^*\mcl P \mcl T \le \norm{\mcl P}_{\mcl L(L_2)} \mcl T^* \mcl T$ and hence by Lem.~\ref{lem:LPI_bounding} $V$ is PDE bounded.  If $\mcl T \dot{\mbf x}(t)=\mcl A \mbf x(t)$, we have that $\mcl T^* \mcl P \mcl A+\mcl A^* \mcl P \mcl T\le 0$ implies
\begin{align*}
\frac{d}{dt}V(\mbf x(t))&=\ip{\mcl T\dot{\mbf x}(t)}{\mcl P\mcl T \mbf x(t)}+\ip{\mcl T\mbf x(t)}{\mcl P\mcl T \dot{\mbf x}(t)}\\
&=\ip{\mcl A\mbf x(t)}{ \mcl P \mcl T\mbf x(t)}+\ip{\mcl T \mbf x(t)}{\mcl P\mcl A \mbf x(t)}\le 0
\end{align*}
and hence $\dot V(x)$ is negative semidefinite. We conclude PDE stability by Lem.~\ref{lem:PIE2PDE_lyapunov}.

For statement 4), let $V(\mbf x)=\ip{\mbf x}{\mcl Q\mcl T \mbf x}$. As in statement 2), $V$ is PIE positive. Next, $\mcl Q \mcl T\le C \mcl T^* \mcl T$ implies $V$ is PDE bounded by Lem.~\ref{lem:LPI_bounding}. As in statement 1) $\dot V(x)$ is negative semidefinite. We conclude PDE to PIE stability by Lem.~\ref{lem:PIE2PDE_lyapunov}.
\end{proof}
\end{note}

\vspace{3mm}

\begin{note}[Proof of Lemma~\ref{lem:exp_stable_LPIs}]\label{note:lemma8_proof}

Recall the statement of Lemma~\ref{lem:exp_stable_LPIs}.

\begin{lem}[LPIs for Exp. Stability]
Given $\mcl T,\mcl A$, the PIE $\mcl T \dot{\mbf x}=\mcl A \mbf x$ is
\begin{enumerate}
\item Exp. PIE to PDE stable if $\mcl Q \mcl T=(\mcl Q \mcl T)^* \ge \epsilon \mcl T^* \mcl T$, $\mcl Q\mcl A+\mcl A^*\mcl Q^*\le -\alpha \mcl Q \mcl T$
\item Exp. PIE stable if $\mcl Q \mcl T=(\mcl Q \mcl T)^* \ge \epsilon I$  and $\mcl Q\mcl A+\mcl A^*\mcl Q^*\le -\alpha \mcl Q \mcl T$
\item Exp. PDE stable if $\mcl P \ge \epsilon I$, $\mcl T^* \mcl P \mcl A+\mcl A^* \mcl P \mcl T\le -\alpha\mcl T^* \mcl P \mcl T$
\item Exp. PDE to PIE stable if $\mcl Q \mcl T=(\mcl Q \mcl T)^* \ge \epsilon I$, $\mcl Q \mcl T\le C \mcl T^* \mcl T$, and $\mcl Q\mcl A+\mcl A^*\mcl Q^*\le -\alpha \mcl Q \mcl T$
  \end{enumerate}
  for some $\mcl Q,\mcl P \in \Pi_3$ and $\epsilon,C,\alpha>0$.
\end{lem}

\begin{proof}
For statement 1), let $V(\mbf x)=\ip{\mbf x}{\mcl Q\mcl T \mbf x}$. Then as in the proof of Lemma~\ref{lem:PIE2PDE_lyapunov_LPI}, $V$ is PDE positive and PIE bounded. As in the proof of Lemma~\ref{lem:PIE2PDE_lyapunov_LPI}, if $\mcl T \dot{\mbf x}(t)=\mcl A \mbf x(t)$, we have
\[
\frac{d}{dt}V(\mbf x(t))=\ip{\mcl Q\mcl A\mbf x(t)}{ \mbf x(t)}+\ip{\mbf x(t)}{\mcl Q\mcl A \mbf x(t)}
\]
and since $\mcl Q\mcl A+\mcl A^*\mcl Q^*\le -\alpha \mcl Q \mcl T$, Lemma~\ref{lem:LPI_bounding} implies $\dot V$ is Lyapunov negative. We conclude Exp. PIE to PDE stability by Lem.~\ref{lem:PIE2PDE_exponential}.

For statement 2), let $V(\mbf x)=\ip{\mbf x}{\mcl Q\mcl T \mbf x}$. Then as in the proof of Lem.~\ref{lem:PIE2PDE_lyapunov_LPI}, $V$ is PIE positive and PIE bounded. As in the proof of statement 1), $\dot V$ Lyapunov negative. We conclude Exp. PIE stability by Lem.~\ref{lem:PIE2PDE_exponential}.

For statement 3), let $V(\mbf x)=\ip{\mbf x}{\mcl T^* \mcl P \mcl T\mbf x}$. Then as in the proof of Lemma~\ref{lem:PIE2PDE_lyapunov_LPI}, $V$ is PDE positive and PDE bounded.  As in the proof of Lemma~\ref{lem:PIE2PDE_lyapunov_LPI}, if $\mcl T \dot{\mbf x}(t)=\mcl A \mbf x(t)$, we have
\[
\frac{d}{dt}V(\mbf x(t))=\ip{\mcl A\mbf x(t)}{ \mcl P \mcl T\mbf x(t)}+\ip{\mcl T \mbf x(t)}{\mcl P\mcl A \mbf x(t)}
\]
and since  $\mcl T^* \mcl P \mcl A+\mcl A^* \mcl P \mcl T\le -\alpha\mcl T^* \mcl P \mcl T$, Lemma~\ref{lem:LPI_bounding} implies $\dot V$ is Lyapunov negative. We conclude Exp. PDE stability by Lem.~\ref{lem:PIE2PDE_exponential}.

For statement 4), let $V(\mbf x)=\ip{\mbf x}{\mcl Q\mcl T \mbf x}$. As in statement 1), $V$ is PIE positive and $\mcl Q \mcl T\le C \mcl T^* \mcl T$ implies $V$ is PDE bounded by Lem.~\ref{lem:LPI_bounding}. As in statement 1) $\dot V$ is Lyapunov negative. We conclude Exp. PDE to PIE stability by Lem.~\ref{lem:PIE2PDE_exponential}.
\end{proof}
\end{note}

\vspace{3mm}

\begin{note}\label{note:exp_LPI_alt}
If we wish to avoid bilinearity in decision variables $\alpha$ and $\mcl Q$ or $\mcl P$ present in the LPI conditions of Lem.~\ref{lem:exp_stable_LPIs}, we may replace the Lyapunov negativity conditions in Lem.~\ref{lem:exp_stable_LPIs} with PIE or PDE negativity. Recalling the relationship between Lyapunov, PDE, and PIE negativity, we have the following.
\begin{lem}[Alternative LPIs for Exp. stability]\label{lem:exp_stable_LPIs_v2}
Given $\mcl T,\mcl A$, the PIE $\mcl T \dot{\mbf x}=\mcl A \mbf x$ is
  \begin{enumerate}
\item Exp. PIE to PDE stable if  $\mcl Q \mcl T=(\mcl Q \mcl T)^* \ge \epsilon \mcl T^* \mcl T$, $\mcl Q\mcl A+\mcl A^*\mcl Q^*\le -\alpha I$
\item Exp. PIE stable if $\mcl Q \mcl T=(\mcl Q \mcl T)^* \ge \epsilon I$  and $\mcl Q\mcl A+\mcl A^*\mcl Q^*\le -\alpha I$
\item Exp. PDE stable if $\mcl P \ge \epsilon I$, $\mcl T^* \mcl P \mcl A+\mcl A^* \mcl P \mcl T\le -\alpha\mcl T^* \mcl T$
\item Exp. PDE to PIE stable if $\mcl Q \mcl T=(\mcl Q \mcl T)^* \ge \epsilon I$, $\mcl Q \mcl T\le C \mcl T^* \mcl T$, and $\mcl Q\mcl A+\mcl A^*\mcl Q^*\le -\alpha I$ 
  \end{enumerate}
  for some $\mcl Q,\mcl P \in \Pi_3$ and $\epsilon,C,\alpha>0$.
\end{lem}
\begin{proof}
Statements 1) and 2) replace $\dot V$ Lyapunov negative with $V$ PIE bounded and $\dot V$ PIE negative. $V$ PIE bounded and $\dot V$ PIE negative imply $\dot V$ Lyapunov negative by Lem.~\ref{lem:LPI_bounding} since $\mcl D \le -\alpha I$ and $\mcl P\le C I$ imply
\[
\mcl D \le -\alpha I\le -\frac{\alpha}{C}\mcl P. 
\]
Statements 3) and 4) replace $\dot V$ Lyapunov negative with $V$ PDE bounded and $\dot V$ PDE negative. $V$ PDE bounded and $\dot V$ PDE negative imply $\dot V$ Lyapunov negative by Lem.~\ref{lem:LPI_bounding} since $\mcl D \le -\alpha \mcl T^*\mcl T$ and $\mcl P\le C \mcl T^* \mcl T$ imply
\[
\mcl D \le -\alpha \mcl T^*\mcl T\le -\frac{\alpha}{C}\mcl P. 
\]
Thus all statements in Lemma~\ref{lem:exp_stable_LPIs_v2} imply the corresponding statements in Lemma~\ref{lem:exp_stable_LPIs}.
\end{proof}
\end{note}
Note that the alternative condition for Exp. PIE to PDE stable may be considerably stronger than the condition for Exp. PIE to PDE stability in Lem.~\ref{lem:exp_stable_LPIs}.
\vspace{3mm}

\begin{note}[Proof of Lemma~\ref{lem:FE_stability_LPI}]\label{note:lemma10_proof}

Recall the statement of Lemma~\ref{lem:FE_stability_LPI}.

\begin{lem}
Given $\mcl T,\mcl A$, the PIE $\mcl T \dot{\mbf x}=\mcl A \mbf x$ is
  \begin{enumerate}
\item FE PIE to PDE stable if $\mcl Q \mcl T=(\mcl Q \mcl T)^* \ge 0$, $\mcl Q\mcl A+\mcl A^*\mcl Q^*\le -\alpha \mcl T^* \mcl T$.
\item FE PIE stable if $\mcl Q \mcl T=(\mcl Q \mcl T)^* \ge 0$  and $\mcl Q\mcl A+\mcl A^*\mcl Q^*\le -\alpha I$.
\item FE PDE stable if $\mcl P \ge 0$, $\mcl T^* \mcl P \mcl A+\mcl A^* \mcl P \mcl T\le -\alpha \mcl T^* \mcl T$.
\item FE PDE to PIE stable if $\mcl Q \mcl T=(\mcl Q \mcl T)^* \ge 0$, $\mcl Q \mcl T\le C \mcl T^* \mcl T$, and $\mcl Q\mcl A+\mcl A^*\mcl Q^*\le -\alpha I$.
  \end{enumerate}
  for some $\mcl Q,\mcl P \in \Pi_3$ and $\epsilon,C,\alpha>0$.
\end{lem}
\begin{proof}
For statement 1), let $V(\mbf x)=\ip{\mbf x}{\mcl Q\mcl T \mbf x}$. Then $\mcl Q \mcl T=(\mcl Q \mcl T)^* \ge 0$ implies $V$ is positive semidefinite. Furthermore, as in the proof of Lemma~\ref{lem:PIE2PDE_lyapunov_LPI}, $V$ is PIE bounded. By Lem.~\ref{lem:LPI_bounding}, $\mcl Q\mcl A+\mcl A^*\mcl Q^*\le -\alpha \mcl T^* \mcl T$, implies $\dot V$ is PDE negative. We conclude FE PIE to PDE stability by Lem.~\ref{lem:FE_stability}.

For statement 2), let $V(\mbf x)=\ip{\mbf x}{\mcl Q\mcl T \mbf x}$. Then as in the proof of statement 1), V is positive semidefinite and PIE bounded. By Lem.~\ref{lem:LPI_bounding}, $\mcl Q\mcl A+\mcl A^*\mcl Q^*\le -\alpha I$ implies $\dot V$ is PIE negative. We conclude FE PIE stability by Lem.~\ref{lem:FE_stability}.

For statement 3), let $V(\mbf x)=\ip{\mbf x}{\mcl T^* \mcl P \mcl T\mbf x}$. Then $\mcl P \ge 0$ implies $V$ is positive semidefinite. Furthermore, as in the proof of Lemma~\ref{lem:PIE2PDE_lyapunov_LPI}, $V$ PDE bounded. By Lem.~\ref{lem:LPI_bounding}, $\mcl T^* \mcl P \mcl A+\mcl A^* \mcl P \mcl T\le -\alpha \mcl T^* \mcl T$ implies $\dot V$ is PDE negative. We conclude FE PDE stability by Lem.~\ref{lem:FE_stability}.

For statement 4), let $V(\mbf x)=\ip{\mbf x}{\mcl Q\mcl T \mbf x}$. Then as in the proof of statement 1), V is positive semidefinite and PIE bounded.  As in statement 2) $\dot V$ is PIE negative. We conclude FE PDE to PIE stability by Lem.~\ref{lem:FE_stability}.
\end{proof}

\end{note}

\end{document}